\documentclass[a4paper,12pt]{article}
\usepackage{amsmath,amssymb,amsthm,graphics,latexsym,amsfonts}
\usepackage{fancyhdr}
\usepackage{color}
\usepackage{graphics}
\usepackage{epstopdf}
\usepackage{amssymb}
\usepackage{cite}
\usepackage{hyperref}
\usepackage{tikz}
\usepackage{cleveref}
\usepackage{mathrsfs}

\hypersetup{
    colorlinks=true,
    linkcolor=cyan,
    filecolor=cyan,
    urlcolor=cyan,
    citecolor=cyan,
}

\title{\Large On the maximum partial-dual genus of a planar graph}
\author{ {Jiaying Chen\,, Xian'an Jin\,, Gang Zhang}\vspace{2mm}\\
	\small  School of Mathematical Sciences, Xiamen University,\\
	\small  Xiamen, Fujian 361005, P.R. China\\
}
\date{\small E-mails: jia\_ying\_chen@163.com, xajin@xmu.edu.cn, gzh\_ang@163.com}

\newtheorem{theorem}{Theorem}[section]
\newtheorem{lemma}[theorem]{Lemma}
\newtheorem{corollary}[theorem]{Corollary}
\newtheorem{claim}{Claim}[section]

\newtheorem{conjecture}[theorem]{Conjecture}
\newtheorem{proposition}[theorem]{Proposition}

\usepackage{indentfirst}

\newcommand{\vertex}{\node[vertex]}
\tikzstyle{vertex}=[circle, draw, inner sep=0pt, minimum size=6pt]

\newenvironment{theorembis}[1]
 {\renewcommand{\thetheorem}{\ref{#1}}%
 \addtocounter{theorem}{-1}%
 \begin{theorem}}
 {\end{theorem}}

\begin{document}

\maketitle

\small \noindent{\bfseries Abstract} Let $G$ be an embedded graph and $A$ an edge subset of $G$. The partial dual of $G$ with respect to $A$, denoted by $G^A$, can be viewed as the geometric dual $G^*$ of $G$ over $A$. If $A=E(G)$, then $G^A=G^*$. Denote by $\gamma(G^A)$ the genus of the embedded graph $G^A$. The maximum partial-dual genus of $G$ is defined as $$^\partial\gamma_{M}(G):=\max_{A \subseteq E(G)}\gamma(G^A).$$
For any planar graph $G$, it had been proved that $^\partial\gamma_{M}(G)$ does not rely on the embeddings of $G$. In this paper, we further prove that if $G$ is a connected planar graph of order $n\geq 2$, then $^{\partial}\gamma_{M}(G)\geq \frac{n-n_2-2n_1}{2}+1$, where $n_i$ is the number of vertices of degree $i$ in $G$. As a consequence, if $G$ is a connected planar graph of order $n$ with minimum degree at least 3, then $^{\partial}\gamma_{M}(G) \geq \frac{n}{2}+1$. Denote by $G^c$ the complement of a graph $G$ and by $\chi(G^c)$ the chromatic number of $G^c$. Moreover, we prove that if $G \ncong K_4$ is a $\lambda$-edge-connected planar graph of order $n$, then $^{\partial}\gamma_{M}(G) \geq f(n,\lambda,\chi(G^c))$, where $f(n,\lambda,\chi(G^c))$ is a function of $n$, $\lambda$ and $\chi(G^c)$. The first lower bound is tight for any $n$, and the second lower bound is tight for some 3-edge-connected graphs. 

\section {\large Introduction}\label{sec1}

Unless otherwise stated, the graphs considered in this paper are all simple and finite, and terms and notations without explicit description are referred to \cite{Bondy2008,Ellis-Monaghan2013}. Let $k \geq 0$ be an integer. A graph $G$ is {\it $k$-connected} (resp., {\it $k$-edge-connected}) if there exist $k$ pairwise internally vertex-disjoint (resp., edge-disjoint) $uv$-paths between any two distinct vertices $u$ and $v$ in $G$. The maximum value of $k$ is called the {\it connectivity} and {\it edge connectivity} of $G$, denoted by $\kappa(G)$ and $\lambda(G)$, respectively. Every $k$-connected graph must be $k$-edge-connected. We use $P_n$, $C_n$, $K_n$ and $nK_1$ to denote the path, cycle, complete graph and empty graph of order $n$, and use $c(G)$ and $G^c$ to denote the number of the components of $G$ and the complement of $G$, respectively. For any integer $l \geq 1$, we denote $[l]:=\{1,2,\ldots,l\}$. An {\it $l$-coloring} of a graph $G$ is a mapping from $V(G)$ to the color set $[l]$ such that no two adjacent vertices have the same color in $G$. The minimum value of $l$ is called the {\it chromatic number} of $G$, denoted by $\chi(G)$.

A {\it surface} $\mathit{\Sigma}$ is a compact and connected 2-manifold with no boundaries. A graph $G$ is said to be {\it embedded} on $\mathit{\Sigma}$ if it is drawn on $\mathit{\Sigma}$ so that its edges intersect only at the common vertices. A graph embedding in this paper is always assumed to be a cellular embedding. The {\it maximum (orientable) genus} of a connected graph $G$, denoted by $\gamma_M(G)$, is the maximum integer $g$ such that there exists an embedding of $G$ on an orientable surface with genus $g$. Since every embedding of $G$ has at least one face, and by Euler’s formula, one can easily obtain that $\gamma_M(G) \leq \left\lfloor\frac{\beta(G)}{2}\right\rfloor$, where $\beta(G)$ is the {\it Betti number} of $G$. We say that $G$ is {\it upper-embeddable} if $\gamma_M(G) = \left\lfloor\frac{\beta(G)}{2}\right\rfloor$.

For any edge or vertex subset $A$ of a (connected) graph $G$, we denote by $G[A]$ the {\it induced subgraph of $G$ by $A$}, and by $G-A$ the {\it graph obtained by deleting $A$ from $G$}. Let $\mathcal{T}$ be the family of spanning trees of $G$. For each $T \in \mathcal{T}$, let $\xi(G-E(T))$ be the number of odd components of $G-E(T)$ (that is, the components with odd number of edges). Moreover, define $$\xi(G):=\min_{T \in \mathcal{T}}\xi(G-E(T)).$$
Notably, the maximum genus of $G$ can be determined by the Xuong formula.

\begin{theorem}(Xuong \cite{Xuong1979}).\label{th1.1}
	If $G$ is a connected graph, then $\gamma_{M}(G)=\frac{\beta(G)-\xi(G)}{2}$. Moreover, $G$ is upper-embeddable if and only if $\xi(G) \in \{0,1\}$.
\end{theorem}

Furthermore, let $A$ be an edge subset of a connected graph $G$, and let $b(G-A)$ be the number of components of $G-A$ with odd Betti number. Nebesk\'{y} \cite{Nebesky1981} gave another characterization of the maximum genus of a connected graph.

\begin{theorem}(Nebesk\'{y} \cite{Nebesky1981}).\label{th1.2}
	If $G$ is a connected graph, then $$\xi(G)=\max \limits_{A\subseteq E(G)} \{c(G-A)+b(G-A)-|A|-1\}.$$ Moreover, $G$ is upper-embeddable if and only if $c(G-A)+b(G-A)-2\leq |A|$ for every edge subset $A\subseteq E(G)$.
\end{theorem}

Let $G^*$ be the geometric dual of an embedded graph $G$, and $A \subseteq E(G)$. In 2009, Chmutov \cite{Chmutov2009} introduced the partial duals of $G$ with respect to $A$, denoted by $G^A$ as usual. Note that $G^{\emptyset}=G$ and $G^{E(G)}=G^*$, and $G^A$ may not be simple. The {\it maximum partial-dual (orientable) genus} $^\partial\gamma_{M}(G)$ of $G$ is defined as $$^\partial\gamma_{M}(G):=\max_{A \subseteq E(G)}\gamma(G^A),$$ where $\gamma(G^A)$ is the genus of the embedded graph $G^A$.

Inspired by the results of Theorems \ref{th1.1} and \ref{th1.2} about the maximum genus of a connected graph, Chen and Chen \cite{Chen&Chen2022}, and Qin and Chen \cite{Qin2024} studied the maximum partial-dual genus of a connected planar graph. Following their notations, for a spanning tree $T$ and an edge subset $A$ of a connected graph $G$, let $$x_G(T):=c(G-E(T)) \text{~~and~~} y_G(A):=2c(G-A)-|A|-1,$$
and moreover, let $$x_G:=\min_{T \in \mathcal{T}}x_G(T) \text{~~and~~} y_G:=\max_{A \subseteq E(G)}y_G(A).$$

\begin{theorem}(Chen and Chen \cite{Chen&Chen2022}).\label{th1.3}
	If $G$ is a connected planar graph of order $n$, then $^\partial\gamma_{M}(G)=n-x_G$.
\end{theorem}

Actually, the graph parameter $x_G$ is also called the {\it decay number} of $G$, which was introduced by \v{S}koviera \cite{Skoviera1992}. In 1995, Nebesk\'{y} \cite{Nebesky1995} proved that $x_G=y_G$ for any connected graph $G$. Together with Theorem \ref{th1.3}, Theorem \ref{th1.5} is immediate.

\begin{theorem}(Nebesk\'{y} \cite{Nebesky1995}).\label{th1.4}
	If $G$ is a connected graph, then $x_G=y_G$.
\end{theorem}

\begin{theorem}(Nebesk\'{y} \cite{Nebesky1995}, Chen and Chen \cite{Chen&Chen2022}).\label{th1.5}
	If $G$ is a connected planar graph of order $n$, then $^\partial\gamma_{M}(G)=n-y_G$.
\end{theorem}

By Theorems \ref{th1.3} and \ref{th1.5}, we know that the maximum partial-dual genus of a planar graph $G$ does not rely on the embeddings of $G$.

\vspace{3mm}

The topological invariant of graphs, maximum genus, has received a great deal of attentions over a long period of time. There are numerous research and results on the relationship between the maximum genus and other graph parameters, such as connectivity \cite{Chen&Archdeacon1996,Ouyang2010,Archdeacon2015}, diameter \cite{Skoviera1991,Huang&Liu1999}, girth \cite{Huang1999,Li2000}, independence number \cite{Huang&Liu1997,Huang&Liu2000}, domination number \cite{Deng2001}, decycling number \cite{Long2018} and so on. In particular, Chen and Kanchi \cite{Chen&Kanchi1996} proved that if $G$ is a connected graph with minimum degree $\delta(G) \geq 3$, then $\gamma_{M}(G) \geq \frac{\beta(G)}{4}$. Huang \cite{Huang2003} proved that if $G$ is a connected graph with edge connectivity $\lambda(G) \leq 3$, then 
\begin{displaymath}
	\gamma_{M}(G)\geq \left\{
	\begin{array}{ll}
		\frac{\beta(G)-\chi(G^c)}{2},& \text{if}\ \lambda(G)=1,\\
		\frac{\beta(G)-\max\{\chi(G^c)-1,1\}}{2},& \text{if}\ \lambda(G)=2,\\
		\frac{\beta(G)-\max\{\left\lfloor\frac{\chi(G^c)}{2}\right\rfloor-1,1\}}{2},& \text{if}\ \lambda(G)=3.
	\end{array} \right.
\end{displaymath}
We remark a well-known result in \cite{Kundu1974} that if $G$ is a 4-edge-connected graph, then $G$ contains two edge-disjoint spanning trees. Consequently, we have $\xi(G) \leq x_G = 1$. By Theorem \ref{th1.1}, we know that $G$ is upper-embeddable and $\gamma_M(G) = \left\lfloor\frac{\beta(G)}{2}\right\rfloor$, and by Theorem \ref{th1.3}, if $G$ is still planar and of order $n$, then $^{\partial}\gamma_{M}(G)=n-1$.

Motivated by the aforementioned work, we further study the maximum partial-dual genus of a connected planar graph in this paper, and we obtain

\begin{theorem}\label{th1.6}
	If $G$ is a connected planar graph of order $n\geq 2$, then 
	\begin{equation*}
		^{\partial}\gamma_{M}(G)\geq \max \{\frac{n-n_2-2n_1}{2}+1,0\},
	\end{equation*}
    where $n_i$ is the number of vertices of degree $i$ in $G$.
\end{theorem}

\begin{corollary}\label{coro1.7}
	If $G$ is a connected planar graph of order $n$ with $\delta(G) \geq 3$, then $$^{\partial}\gamma_{M}(G) \geq \frac{n}{2}+1.$$
\end{corollary}

\begin{theorem}\label{th1.8}
	If $G \ncong K_4$ is a $\lambda$-edge-connected planar graph of order $n$ where $\lambda \leq 3$, then 
	\begin{displaymath}
		^{\partial}\gamma_{M}(G)\geq \left\{
		\begin{array}{ll}
			\max \{0,n-3\chi{(G^c)}\},& \text{if}\ \lambda=1,\\
			\max \{1,n-3\chi{(G^c)}+1\},& \text{if}\ \lambda=2,\\
			\max \{\left\lceil\frac{n}{2}\right\rceil+1,n-\left \lfloor \frac{3\chi{(G^c)}}{2} \right \rfloor+1\},& \text{if}\ \lambda=3.
		\end{array} \right.
	\end{displaymath}
\end{theorem}

\section {\large Preliminaries}\label{sec2}

In this section, we present some preliminary results that will be used to prove our main results.

\begin{lemma}\label{lem2.1}
	Let $A$ be a minimal edge cut of a connected graph $G$, and $G_1$ and $G_2$ be the two components of $G-A$. Then
	\begin{displaymath}
		y_G= \left\{
		\begin{array}{ll}
			y_{G_1}+y_{G_2},& \text{if}\ |A|=1,\\
			y_{G_1}+y_{G_2}-1,& \text{if}\ |A|=2.
		\end{array} \right.
	\end{displaymath}
\end{lemma}

\begin{proof}
	It is equivalent to prove that $y_G=y_{G_1}+y_{G_2}-|A|+1$ where $|A| \in [2]$. For each $i \in [2]$, let $A_i$ be an edge subset of $G_i$ such that $y_{G_i}=y_{G_i}(A_i)$. Then $A^*=A_1\cup A_2\cup A$ is an edge subset of $G$ satisfying that $|A^*|=|A_1|+|A_2|+|A|$ and $c(G-A^*)=c(G_1-A_1)+c(G_2-A_2)$. Hence,
	\begin{equation*}
		\begin{aligned}
			y_G(A^*)&=2c(G-A^*)-|A^*|-1\\
			&=2c(G_1-A_1)-|A_1|+2c(G_2-A_2)-|A_2|-|A|-1\\
			&=y_{G_1}(A_1)+y_{G_2}(A_2)-|A|+1=y_{G_1}+y_{G_2}-|A|+1.
		\end{aligned}
	\end{equation*}
	
	We claim that $y_G(A^*)=y_G$. Suppose to the contrary that there exists an edge subset $A'$ of $G$ such that $y_G(A^*)<y_G(A')$. Let $A^{'}_i=A'\cap E(G_i)$ for each $i \in [2]$. We distinguish the proof into the following two cases.
	
	\textit{Case 1:} $|A \cap A'|=|A|$, equivalently, we know $A \subseteq A'$. Clearly, $A'=A^{'}_1\cup A^{'}_2\cup A$, and moreover, $|A'|=|A_1'|+|A_2'|+|A|$ and $c(G-A')=c(G_1-A^{'}_1)+c(G_2-A^{'}_2)$. Since $y_G(A^*)=y_{G_1}+y_{G_2}-|A|+1$, we have
	\begin{equation*}
		\begin{aligned}
			y_{G_1}(A_1)+y_{G_2}(A_2)-|A|&+1=y_G(A^*)<y_G(A')=2c(G-A')-|A'|-1\\
			&=2c(G_1-A^{'}_{1})-|A^{'}_{1}|+2c(G_2-A^{'}_{2})-|A^{'}_{2}|-|A|-1\\
			&=y_{G_1}(A^{'}_1)+y_{G_2}(A^{'}_2)-|A|+1.
		\end{aligned}
	\end{equation*}
	So, there exists some $i\in [2]$ such that $y_{G_i}(A_i)<y_{G_i}(A^{'}_i)$, a contradiction.
	
	\textit{Case 2:} $|A \cap A'| \leq |A|-1$, equivalently, there is at least one edge of $A$ not in $A'$. Clearly, $A \not\subseteq A'$, $|A^{'}_1|+|A^{'}_2|\leq|A'|$ and $c(G-A')\leq c(G_1-A^{'}_1)+c(G_2-A^{'}_2)-1$. Since $|A|\leq 2$, we have 
	\begin{equation*}
		\begin{aligned}
			y_{G_1}(A_1)&+y_{G_2}(A_2)-1\leq y_{G_1}(A_1)+y_{G_2}(A_2)-|A|+1=y_G(A^*)<y_G(A')\\
			&=2c(G-A')-|A'|-1\leq 2c(G_1-A^{'}_{1})-|A^{'}_{1}|+2c(G_2-A^{'}_{2})-|A^{'}_{2}|-3\\
			&=y_{G_1}(A^{'}_1)+y_{G_2}(A^{'}_2)-1,
		\end{aligned}
	\end{equation*}
	giving the same contradiction as Case 1. This completes the proof.
\end{proof}

\noindent{\bf{Operation $\mathcal{O}$}.} Let $G_1,G_2,\ldots,G_k$ be $k \geq 2$ vertex-disjoint connected graphs, and $v_i$ be a vertex of $G_i$ for each $i \in [k]$. The operation identifies $v_1,v_2,\ldots,v_k$ to a new vertex $v$, that is, replaces these vertices by a single vertex $v$ in the resulting graph $G$ incident to all the edges which are incident to $v_i$ in $G_i$.

\begin{lemma}\label{lem2.2}
	Let $G$ be the graph obtained from $G_1,G_2,\ldots,G_k$ by the operation $\mathcal{O}$. Then $$y_G=\sum_{i \in [k]}y_{G_i}-k+1.$$
\end{lemma}

\begin{proof}
	Note that $G$ is connected, and $$V(G)=\left(\bigcup_{i \in [k]}V(G_i) \setminus \{v_i\}\right) \cup \{v\} ~\text{and}~ E(G)=\bigcup_{i \in [k]}E(G_i).$$
	
    \noindent For each $i \in [k]$, let $A_i$ be an edge subset of $G_i$ such that $y_{G_i}=y_{G_i}(A_i)$. Then $A=\bigcup_{i \in [k]} A_i$ is an edge subset of $G$ where $|A|=\sum_{i \in [k]}|A_i|$. Moreover, $c(G-A)=\sum_{i \in [k]}c(G_i-A_i)-k+1$. Hence, we have
	\begin{equation*}
		\begin{aligned}
			y_G(A)=2c(G-A)-|A|-1&=2\sum_{i \in [k]}c(G_i-A_i)-2k+2-\sum_{i \in [k]}|A_i|-1\\
			&=\sum_{i \in [k]}(2c(G_i-A_i)-|A_i|-1)-k+1\\
			&=\sum_{i \in [k]}y_{G_i}(A_i)-k+1=\sum_{i \in [k]}y_{G_i}-k+1.
		\end{aligned}
	\end{equation*}
	
	We claim that $y_G=y_G(A)$. Suppose to the contrary that there exists an edge subset $A'\ne A$ of $G$ such that $y_G(A)<y_G(A')$. Let $A^{'}_{i}=A^{'}\cap E(G_i)$ for each $i \in [k]$. Clearly, $|A'|=\sum_{i \in [k]}|A_i'|$ and $c(G-A')=\sum_{i \in [k]}c(G_i-A_i')-k+1$. Since $y_G(A)=\sum_{i \in [k]}y_{G_i}(A_i)-k+1$, we have
	\begin{equation*}
		\begin{aligned}
			\sum_{i \in [k]}y_{G_i}(A_i)-k+1=y_G(A)<y_G(A')=\sum_{i \in [k]}y_{G_i}(A_i')-k+1.
		\end{aligned}
	\end{equation*}
	So, there exists some $j\in [k]$ such that $y_{G_j}(A_j)<y_{G_j}(A^{'}_j)$, a contradiction.
\end{proof}

For any connected planar graph $G$, $x_G \geq 1$, and by Theorem \ref{th1.3}, $^\partial\gamma_{M}(G) \leq n-1$. Like the upper embedability of graphs, Qin and Chen \cite{Qin2024} defined maximal partial duals of planar graphs. A connected planar graph $G$ is said to be {\it maximal partial-dual} if $^\partial\gamma_{M}(G)=n-1$. By Theorem \ref{th1.4}, $G$ is maximal partial-dual if and only if $x_G=y_G=1$. As noted in Section \ref{sec1}, every 4-edge-connected planar graph $G$ has $x_G=1$, and thus, is maximal partial-dual. In \cite{Qin2024}, Qin and Chen gave a structural characterization theorem for non-maximal partial-dual planar graphs. The following Theorem \ref{thm2.3} is an extension of this characterization theorem.

Let $G$ be a connected graph and $A \subseteq E(G)$. Assume that $F_1, F_2, \ldots, F_k$ are the $k \geq 2$ distinct components of $G-A$, and we denote by $E(F_1, F_2, \ldots, F_k)$ the set of the edges joining these $k$ components in $G$.

\begin{theorem}\label{thm2.3} 
	Let $G$ be a connected graph and $A$ an edge subset of $G$ such that $y_G=y_G(A)$. Then the following properties hold:
	
	(i) If $y_G=1$, then $|A| \geq 0$, $c(G-A)\geq 1$ and $|A|= 2c(G-A)-2$. If $y_G \geq 2$, then $|A| \geq 1$, $c(G-A)\geq 2$ and $|A|\leq 2c(G-A)-3$;
	
	(ii) For any component $F$ of $G-A$, $F$ is a vertex-induced subgraph of $G$, and $y_F=1$. In particular, if $F$ is the unique component of $G-A$, then $F = G$, $y_F=y_G=1$, $A=\emptyset$ and $c(G-A)=1$;
	
	(iii) For any $k$ distinct components $F_1, F_2, \ldots, F_k~(k\geq 2)$ of $G-A$, $|E(F_1, F_2, \ldots, F_k)|\leq 2k-2$, especially $|E(F_i, F_j)| \leq 2$ for $1 \leq i < j \leq k$. If $A$ is further a smallest edge subset of $G$ with $y_G=y_G(A)$, then $|E(F_1, F_2, \ldots, F_k)|\leq 2k-3$ and $|E(F_i, F_j)| \leq 1$.
\end{theorem}

\begin{proof}
	(i) It is clear that $|A| \geq 0$ and $c(G-A) \geq 1$ always hold. By the definition of $y_G(A)$, we have $1=y_G=y_G(A)=2c(G-A)-|A|-1$, that is, $|A|= 2c(G-A)-2$. If $y_G \geq 2$, then $y_G(A)=y_G\geq 2 >1=y_G(\emptyset)$, and thus, $|A| \geq 1$. Since $y_G(A)=2c(G-A)-|A|-1 \geq 2$, $|A|\leq 2c(G-A)-3$ and $c(G-A) \geq \frac{3+|A|}{2} \geq 2$.
	
	(ii) We first prove that each component of $G-A$ is a vertex-induced subgraph of $G$. Suppose that $F$ is not a such component of $G-A$. Then $F$ is not a vertex-induced subgraph of $G[V(F)]$, and there exists an edge $e \in E(G[V(F)]) \setminus E(F)$. Clearly, $e\in A \subseteq E(G)$. Let $A'=A\setminus \{e\}$. It is easy to see that $c(G-A')=c(G-A)$ and $|A'|=|A|-1$, and thus, we have 
	\begin{equation*}
		y_G(A')=2c(G-A')-|A'|-1
		=2c(G-A)-|A|+1-1=y_G(A)+1=y_G+1,
	\end{equation*}
	which contradicts that $y_G(A') \leq y_G$.
	
	Now we prove that $y_F=1$ for any $F$. Suppose that there exists a component $F$ of $G-A$ such that $y_F \geq 2$. Let $A_F$ be an edge subset of $F$ such that $y_F=y_F(A_{F})$. Since $F$ is a vertex-induced graph of $G$, $A \cap A_{F}=\emptyset$. Let $A'=(A_{F} \cup A) \subseteq E(G)$. Then  $c(G-A')=c(G-A)-1+c(F-A_{F})$ and $|A'|=|A_{F}|+|A|$. Hence, we can obtain a contradiction that
	\begin{equation*}
		\begin{aligned}
			y_G\geq y_G(A')&=2c(G-A')-|A'|-1\\
			&=2(c(G-A)-1+c(F-A_{F}))-(|A_{F}|+|A|)-1\\
			&=2c(G-A)-|A|-1+2c(F-A_{F})-|A_{F}|-1-1\\
			&=y_G(A)+y_{F}(A_{F})-1=y_G+y_{F}-1\geq y_G+1.
		\end{aligned}
	\end{equation*}
	
	(iii) Suppose that there are $k$ distinct components $F_1, F_2, \ldots, F_k~(k\geq 2)$ of $G-A$ such that $|E(F_1, F_2, \ldots, F_k)| \geq 2k-1 \geq 3$. Clearly, $E(F_1, F_2, \ldots, F_k) \subseteq A$. Let $A'=A \setminus E(F_1, F_2, \ldots, F_k)$. Note that $G[\cup_{i \in [k]}V(F_i)]$ may be connected or disconnected. Then $c(G-A')\geq c(G-A)-k+1$ and $|A'|=|A|-|E(F_1, F_2, \ldots, F_k)|$. Hence, we obtain a contradiction that
	\begin{equation*}
		\begin{aligned}
			y_G\geq y_G(A')&=2c(G-A')-|A'|-1\\
			&\geq 2(c(G-A)-k+1)-(|A|-|E(F_1, F_2, \ldots, F_k)|)-1\\
			&=2c(G-A)-|A|-1-2k+2+|E(F_1, F_2, \ldots, F_k)|\\
			&\geq y_G(A)+1=y_G+1.
		\end{aligned}
	\end{equation*}
	
	Let $A$ be a smallest edge subset of $G$ such that $y_G(A)=y_G$. Similar to the proof above, we now suppose $|E(F_1, F_2, \ldots, F_k)| \geq 2k-2 \geq 2$, and let $A'=A \setminus E(F_1, F_2, \ldots, F_k)$ as well. We obtain
	\begin{equation*}
		\begin{aligned}
			y_G\geq y_G(A') \geq 2c(G-A)-|A|-1-2k+2+|E(F_1, F_2, \ldots, F_k)|\geq y_G(A)=y_G,
		\end{aligned}
	\end{equation*}
	implying that $y_G(A')=y_G=y_G(A)$, a contradiction to the minimality of $A$. This completes the proof of the theorem.
\end{proof}

\noindent\textit{Remark.} Compared with the Qin-Chen result \cite{Qin2024}, Theorem \ref{thm2.3} suggests that the edge subsets $A$ with $y_G=y_G(A)$ in general graphs $G$ satisfy some nice properties, not only for planar graphs $G$ with $y_G \geq 2$ and existence of $A \subseteq E(G)$. Beyond that, we add more structural characterizations.

\begin{figure}[h!]
	\begin{center}
		\begin{tikzpicture}[scale=.35]
			\tikzstyle{vertex}=[circle, draw, inner sep=0pt, minimum size=6pt]
			\tikzset{vertexStyle/.append style={rectangle}}
			
			\vertex (1) at (0,0) [scale=0.5,fill=black] {};
			\vertex (2) at (8,0) [scale=0.5,fill=black] {};
			\vertex (3) at (4,7) [scale=0.5,fill=black] {};
			\vertex (4) at (4,2.5) [scale=0.5,fill=black] {};
			
			\vertex (5) at (16,0) [scale=0.5,fill=black] {};
			\vertex (6) at (24,0) [scale=0.5,fill=black] {};
			\vertex (7) at (20,7) [scale=0.5,fill=black] {};
			\vertex (8) at (20,2.5) [scale=0.5,fill=black] {};
			
			\path
			(1) edge (2)
			(2) edge (3)
			(3) edge (4) 
			(2) edge (4)
			(4) edge (1)
			(1) edge (3)
			
			(5) edge (6)
			(6) edge (7)
			(7) edge (8) 
			(6) edge (8)
			(8) edge (5)
			(5) edge (7)
			;
			
			\path[draw=black] (2) edge (5) ;
			\node ($e_2$) at (12,-1) {$e_2$};
			\draw[dashed] (3) edge (7);
			\node ($e_1$) at (12,8) {$e_1$};
			
		\end{tikzpicture}
	\end{center}
	{\footnotesize \centerline{{\bf Fig. 1.}~ The graphs $G$ and $G'$, where $G'=G-e_1$.\hypertarget{Fig1}}}
\end{figure}

The graphs $G$ and $G'$ in Fig. \hyperlink{Fig1}{1} can illustrate Theorem \ref{thm2.3} well. Since $y_{K_4}=1$, and by Lemma \ref{lem2.1}, we have $y_{G}=y_{K_4}+y_{K_4}-1=1$ and $y_{G'}=y_{K_4}+y_{K_4}=2$. Let $A_1=\emptyset$ and $A_2=\{e_1,e_2\}$. Then $y_G(A_1)=y_G(A_2)=y_G=1$ and $A_1$ is the unique smallest such edge subset of $G$. For any edge subset $A$ of $G$ with $y_G=y_G(A)$, $|A| \geq |A_1|=0$ and $c(G-A) \geq c(G-A_1)=1$ and $|A_i|=2c(G-A_i)-2$ for $i \in [2]$. Let $A_1'=\{e_2\}$ and $A_2'=E(G')$. Then $y_{G'}(A_1')=y_{G'}(A_2')=y_{G'}=2$ and $A_1'$ is the unique smallest such edge subset of $G'$. It is easy to verify that Theorem \ref{thm2.3} (i) holds for $G'$ also. Moreover, (ii) holds for $G$ and $A_i$, and $G'$ and $A_i'$, and specially, $G-A_1=G$. The pairs $(G,A_2)$ and $(G',A_1')$ attain the bounds on the sizes of the edge cuts in Theorem \ref{thm2.3} (iii), respectively.

\vspace{3mm}
In the first section, we note that

\begin{proposition}\label{prop2.4}
	If $G$ is a 4-edge-connected graph, then $x_G=y_G=1$.
\end{proposition}

Let $G$ be a connected graph and $A \subseteq E(G)$, and let $F_1, F_2, \ldots, F_{c(G-A)}$ be the components of $G-A$. We define the graph $G_A$ obtained from $G$ by identifying the vertices of $F_i$ to a new vertex $u_i$ for $i \in [c(G-A)]$, and then deleting the loops and multiple edges. Clearly, $V(G_A)=\cup_{i \in [c(G-A)]}\{u_i\}$, $E(G_A) \subseteq A$, and $G_A$ is simple and connected. Further, if $A$ is a smallest edge subset of $G$ with $y_G(A)=y_G \geq 2$, then by Theorem \ref{thm2.3} (ii) and (iii), $F_i$ is a vertex-induced subgraph of $G$, and $|E(F_j,F_k)| \leq 1$ for $1\leq j<k \leq c(G-A)$ where $c(G-A) \geq 2$ by (i). Now we create no multiple edges during operation above, and $E(G_A)=A$.

\begin{lemma}\label{lem2.5}
	Let $G$ be a $\lambda$-edge-connected graph where $\lambda \leq 3$. If $A$ is a smallest edge subset of $G$ such that $y_G(A)=y_G \geq 2$, then
	\begin{displaymath}
		y_G\leq \left\{
		\begin{array}{ll}
			c(G-A),& \text{if}\ \lambda=1,\\
			c(G-A)-1,& \text{if}\ \lambda=2,\\
			\left \lfloor \frac{c(G-A)}{2} \right \rfloor-1,& \text{if}\ \lambda=3.
		\end{array} \right.
	\end{displaymath}
\end{lemma}

\begin{proof}
	Note that $c(G-A)=|V(G_A)|$ and $|A|=|E(G_A)|$. Then $G_A$ is connected implying $|E(G_A)| \geq |V(G_A)|-1$. Thus, we have 
	\begin{equation*}
		\begin{aligned}
			y_G=y_G(A)&=2c(G-A)-|A|-1=2|V(G_A)|-|E(G_A)|-1\\
			&\leq 2|V(G_A)|-|V(G_A)|+1-1=|V(G_A)|=c(G-A).
			\end{aligned}
		\end{equation*}
	
	If $\lambda \geq 2$, then $d_{G_A}(u) \geq 2$ for each $u\in V(G_A)$. Hence, $2|A|=2|E(G_A)|=\sum_{u \in V(G_A)}d_{G_A}(u) \geq 2|V(G_A)|=2c(G-A)$, that is, $|A| \geq c(G-A)$, and we have
	\begin{equation*}
		y_G=y_G(A)=2c(G-A)-|A|-1\leq 2c(G-A)-c(G-A)-1=c(G-A)-1.
	\end{equation*}
	
	Similarly, if $\lambda \geq 3$, then $d_{G_A}(u) \geq 3$ for each $u\in V(G_A)$, and then $2|A| \geq 3c(G-A)$. Thus, we have
	\begin{equation*}
		y_G=y_G(A)=2c(G-A)-|A|-1 \leq 2c(G-A)-\frac{3c(G-A)}{2}-1=\frac{c(G-A)}{2} -1.
	\end{equation*}
    Since $y_G$ is an integer, $y_G \leq \left\lfloor\frac{c(G-A)}{2}-1\right\rfloor=\left\lfloor\frac{c(G-A)}{2}\right\rfloor-1$. The result follows.
\end{proof}

The following theorem due to \v{S}koviera \cite{Skoviera1992} is needed in our proofs.

\begin{theorem}(\v{S}koviera \cite{Skoviera1992}).\label{th2.6}
	If $G$ is a connected cubic graph of order $n$, then $x_G=\frac{n}{2}-1$.
\end{theorem}

\section{\large Proofs}\label{sec3}

Recall that for $i \in [2]$, $n_i(G)$ denote the number of vertices of degree $i$ in a graph $G$. To prove Theorem \ref{th1.6}, by Theorem \ref{th1.5}, it suffices to prove

\begin{theorem}\label{th3.1}
	If $G$ is a connected graph of order $n\geq 2$, then
	\begin{equation*}
	y_G\leq \min\{\frac{n+n_2(G)+2n_1(G)}{2}-1,n\}.
	\end{equation*}
\end{theorem}

\begin{proof}
	By induction on $n$. If $n\leq 3$, then $G\in\{K_2, P_3, C_3\}$. Clearly, $y_G=n=\frac{n+n_2(G)+2n_1(G)-2}{2}$ when $G \in \{K_2,P_3\}$ and $y_G=2=\frac{n+n_2(G)+2n_1(G)-2}{2}=n-1<n$ when $G \cong C_3$. So, the result is true, and we may assume that $n\geq 4$. If $y_G=1$, then $y_G \leq 1= \frac{4+0+0-2}{2} \leq \frac{n+n_2(G)+2n_1(G)-2}{2}$. So, assume that $y_G\geq 2$. Let $A$ be a smallest edge subset of $G$ such that $y_G=y_G(A)$. By Lemma \ref{lem2.5}, we have $y_G\leq c(G-A)\leq n$. Therefore, it suffices to prove $y_G\leq \frac{n+n_2(G)+2n_1(G)-2}{2}$.

	If $G$ has a cut vertex $v$, then $d_G(v)\geq 2$. Let $G_1', G_2',\ldots,G_k'$ be all distinct components of $G-v$. Then $k \geq 2$. For each $i \in [k]$, let $G_i$ denote the subgraph of $G$ induced by $V(G_i')\cup \{v\}$. It is easy to see that $G$ can be regarded as the resulting graph obtained from $G_1, G_2,\ldots,G_k$ by the operation $\mathcal{O}$ (defined in Section \ref{sec2}), where $V(G)=\bigcup_{i \in [k]}V(G_i)$ and $E(G)=\bigcup_{i \in [k]}E(G_i)$.
	
	If $d_G(v)=2$, then $k=2$ and $d_{G_i}(v)=1$ for $i \in [2]$. By Lemma \ref{lem2.2},
	\begin{equation*}
		y_G=y_{G_1}+y_{G_2}-1.
	\end{equation*}
    \noindent Moreover, we can easily see
	\begin{equation*}
		\begin{gathered}
			n=n(G_1)+n(G_2)-1,\\
			n_2(G)=n_2(G_1)+n_2(G_2)+1,\\
			n_1(G)=n_1(G_1)+n_1(G_2)-2.
		\end{gathered}
	\end{equation*}
	By the induction hypothesis, we have
	\begin{equation*}
		\begin{aligned}
			y_G&\leq \frac{n(G_1)+n_2(G_1)+2n_1(G_1)-2}{2}+\frac{n(G_2)+n_2(G_2)+2n_1(G_2)-2}{2}-1\\
			&=\frac{n+1+n_2(G)-1+2n_1(G)+4-6}{2}\\
			&= \frac{n+n_2(G)+2n_1(G)-2}{2}.
		\end{aligned}
	\end{equation*}
	
	If $d_G(v)\geq 3$, then let $s_j$ be the number of subgraphs $G_i$ of $G$ with $d_{G_i}(v)=j$, where $i \in [k]$ and $j \in [2]$. By Lemma \ref{lem2.2}, we obtain
	\begin{equation*}
		\begin{gathered}
			y_G=y_{G_1}+y_{G_2}+\dots+y_{G_k}-k+1,\\
			n=n(G_1)+n(G_2)+\dots+n(G_k)-k+1,\\
			n_2(G)=n_2(G_1)+n_2(G_2)+\dots+n_2(G_k)-s_2,\\
			n_1(G)=n_1(G_1)+n_1(G_2)+\dots+n_1(G_k)-s_1.
		\end{gathered}
	\end{equation*}
	Since $2 \leq n(G_i)<n$, it follows from the induction hypothesis that
	\begin{equation*}
		\begin{aligned}
			y_G&\leq \sum_{i \in [k]} \frac{n(G_i)+n_2(G_i)+2n_1(G_i)-2}{2}-k+1\\
			&=\frac{n+k-1+n_2(G)+s_2+2n_1(G)+2s_1-2k}{2}-k+1\\
			&= \frac{n+n_2(G)+2n_1(G)+s_2+2s_1-3k+1}{2}.
		\end{aligned}
	\end{equation*}
	
	\noindent Note that we have done if $s_2+2s_1-3k+1 \leq -2$. It remains to consider the case of $s_2+2s_1-3k+1 > -2$, i.e., $k \leq \frac{s_2+2s_1+2}{3}$. Since $2 \leq k$ and $s_1+s_2 \leq k$, we have $4 \leq s_2+2s_1$ and $s_1+2s_2 \leq 2$. Clearly, $s_1 =2$ and $s_2 =0$, and further, we know $k=2$. This contradicts $d_G(v) \geq 3$.

	Now suppose that $G$ is 2-connected. By Proposition \ref{prop2.4}, we may assume that $G$ is non-4-edge-connected; otherwise, $y_G=1$. If $G$ has an edge cut $\{e_1,e_2\}$, then $G-\{e_1,e_2\}$ has two components $G_1$ and $G_2$. For $i \in [2]$, let $e_i=u_iv_i$ with $u_i \in V(G_1)$ and $v_i \in V(G_2)$. Since $G$ is simple and 2-connected, $u_1\ne u_2$ and $v_1\ne v_2$. Let $t_j$ the number of vertices $u$ of degree $j$ in $G_1$ and $G_2$, where $j \in [2]$ and $u \in \{u_1,v_1,u_2,v_2\}$. Clearly, $0 \leq t_j \leq 4$. By Lemma \ref{lem2.1},
	\begin{equation*}
	\begin{gathered}
		y_G=y_{G_1}+y_{G_2}-1,\\
		n=n(G_1)+n(G_2),\\ 
		n_2(G)=n_2(G_1)+n_2(G_2)+t_1-t_2,\\
		n_1(G)=n_1(G_1)+n_1(G_2)-t_1.
	\end{gathered}
    \end{equation*}	

	\noindent Applying the induction hypothesis, since $t_1+t_2\leq 4$, we have
	\begin{equation*}
		\begin{aligned}
			y_G&\leq \frac{n(G_1)+n_2(G_1)+2n_1(G_1)-2}{2}+\frac{n(G_2)+n_2(G_2)+2n_1(G_2)-2}{2}-1\\
			&=\frac{n+n_2(G)-t_1+t_2+2n_1(G)+2t_1-6}{2}\\
			&\leq \frac{n+n_2(G)+2n_1(G)-2}{2}.
			\end{aligned}
	\end{equation*}
	
	Finally, we consider the case that $G$ is 3-edge-connected. For each $u \in V(G)$, $d_G(u) \geq 3$. By Lemma \ref{lem2.5}, $y_G\leq \left \lfloor \frac{c(G-A)}{2} \right \rfloor -1 \leq \left \lfloor \frac{n}{2} \right \rfloor-1 \leq \frac{n}{2}-1$, where $A$ is an edge subset of $G$ with $y_G=y_G(A)$. This completes the proof of the theorem.
\end{proof}

	\noindent\textit{Remark.} The upper bound in Theorem \hyperref[th3.1]{3.1} is sharp, which can be achieved by the cycle $C_n(n\geq 3)$, the path $P_n(n\geq 2)$, and the star $S_n(n\geq 5)$. If $G \cong C_n$, then $y_G=n-1=\frac{n+n+0}{2}-1=\frac{n+n_2(G)+2n_1(G)}{2}-1<n$. If $G \cong P_n$, then $y_G=n=\frac{n+n-2+4}{2}-1=\frac{n+n_2(G)+2n_1(G)}{2}-1$. If $G \cong S_n$, then $y_G=n<\frac{3n}{2}-2=\frac{n+0+2(n-1)}{2}-1=\frac{n+n_2(G)+2n_1(G)}{2}-1$.

\newpage
We now prove Theorem \ref{th1.8}. Recall its statement.

\begin{theorembis}{th1.8}
	If $G \ncong K_4$ is a $\lambda$-edge-connected planar graph of order $n$ where $\lambda \leq 3$, then 
	\begin{displaymath}
		^{\partial}\gamma_{M}(G)\geq \left\{
		\begin{array}{ll}
			\max \{0,n-3\chi{(G^c)}\},& \text{if}\ \lambda=1,\\
			\max \{1,n-3\chi{(G^c)}+1\},& \text{if}\ \lambda=2,\\
			\max \{\left\lceil\frac{n}{2}\right\rceil+1,n-\left \lfloor \frac{3\chi{(G^c)}}{2} \right \rfloor+1\},& \text{if}\ \lambda=3.
		\end{array} \right.
	\end{displaymath}
\end{theorembis}

Note that $G \cong K_n$ if and only if $G^c \cong nK_1$. For this case, we have $\chi(G^c)=1$. If $G \cong K_4$, then $y_G=1$. By Theorem \ref{th1.5}, $^{\partial}\gamma_{M}(G)=n-y_G=4-1=3<4=f(n=4,\lambda=3,\chi(G^c)=1)$, where $f$ is the piecewise function in Theorem \ref{th1.8}. If $G \cong K_n$ and $n \geq 5$, then $G$ is non-planar and 4-edge-connected. By Proposition \hyperref[prop2.4]{2.4}, $y_G=1=\chi(G^c)$. So, we may consider that $G \ncong K_n$ in the next proof.

\begin{theorem}\label{th3.2}
	If $G \ncong K_n$ is a $\lambda$-edge-connected graph of order $n$ where $\lambda \leq 3$, then 
	\begin{displaymath}
		y_G\leq \left\{
		\begin{array}{ll}
			\min \{n,3\chi{(G^c)}\},& \text{if}\ \lambda=1,\\
			\min \{n-1,3\chi{(G^c)}-1\},& \text{if}\ \lambda=2,\\
			\min \{\left\lfloor\frac{n}{2}\right\rfloor-1,\left \lfloor \frac{3\chi{(G^c)}}{2} \right \rfloor-1\},& \text{if}\ \lambda=3.
		\end{array} \right.
	\end{displaymath}
\end{theorem}

\begin{proof}
	Since $G \ncong K_n$, we know $\chi(G^c)\geq 2$. Recall that $y_G \geq 1$. If $y_G=1$, then the result is trivial. Thus, we may assume that $y_G\geq 2$. Let $A$ be a smallest edge subset of $G$ such that $y_G=y_G(A) \geq 2$. By Theorem \hyperref[thm2.3]{2.3} (i), $2 \leq c(G-A) \leq n$, and by Lemma \ref{lem2.5}, we obtain
		\begin{displaymath}
		y_G\leq \left\{
		\begin{array}{ll}
			c(G-A) \leq n,& \text{if}\ \lambda=1,\\
			c(G-A)-1 \leq n-1,& \text{if}\ \lambda=2,\\
			\left\lfloor\frac{c(G-A)}{2}\right\rfloor-1 \leq \left\lfloor\frac{n}{2}\right\rfloor-1,& \text{if}\ \lambda=3.
		\end{array} \right.
	\end{displaymath}
	\noindent So, it suffices to prove $c(G-A)\leq 3\chi(G^c)$. Let $s=\chi(G^c) \geq 2$. Then $V(G)$ can be partitioned into $s$ sets $V_1,V_2,\ldots,V_s$ such that $G[V_i]$ is a complete graph where $i \in [s]$. First, we  present two claims to describe some structural properties of $G$.
	
	\begin{claim}\label{claim3.1}
		Let $i \in [s]$ and $F$ be a component of $G-A$. If $V_i\nsubseteq$ $V(F)$, then $0 \leq |V(F)\cap V_i| \leq 1$.
	\end{claim}
	
	\begin{proof}
		Suppose to the contrary that there are some $i \in [s]$ and component $F$ of $G-A$ such that $V_i\nsubseteq V(F)$ and $|V(F)\cap V_i|\geq 2$. Then, there exist two vertices $u,v\in V(F)\cap V_i$ and one vertex $w\in V(F')\cap V_i$ for another component $F'$ of $G-A$. Since $G[V_i]$ is a complete graph, $uw,vw\in E(G)$, and since $F$ and $F'$ are two distinct components of $G-A$, $uw,vw\in A$. Hence, we have $|E(F,F')|\geq 2$, contradicting $|E(F,F')|\leq 1$ by Theorem \hyperref[thm2.3]{2.3} (iii).
	\end{proof}

	\begin{claim}\label{claim3.2}
		Let $i \in [s]$ and $F$ be a component of $G-A$ where $V(F) \cap V_i \neq \emptyset$. If $V_i\nsubseteq$ $V(F)$, then $2 \leq |V_i| \leq 3$.
	\end{claim}

	\begin{proof}
		Since $V(F) \cap V_i \neq \emptyset$ and $V_i\nsubseteq$ $V(F)$, $|V_i| \geq 2$. Let $V_i:=\{x_1,x_2,\ldots,x_k\}$ with $k\geq 2$. By Claim \ref{claim3.1}, there exist $k$ distinct components $F_1=F,F_2,\ldots,F_k$ of $G-A$ such that $V(F_j)\cap V_i=\{x_j\}$ for each $j\in [k]$. Note that for each $e\in G[V_i]$, $e\in E(F_1, F_2,\ldots,F_k)$.
	Since $G[V_i]$ is a complete graph, and by Theorem \hyperref[thm2.3]{2.3} (iii), we have 
	\begin{equation*}
		\frac{k(k-1)}{2}=|E(G[V_i])|\leq |E(F_1, F_2,...,F_k)|\leq 2k-3,
	\end{equation*}
	which provides $2\leq k \leq 3$, i.e., $2 \leq |V_i| \leq 3$.
	\end{proof}
	
	Now we shall use the two claims above to complete the proof of Theorem \ref{th3.2} by proving $c(G-A)\leq 3s$. Let $\mathcal{F}$ be the family of components of $G-A$, and $\mathcal{V}:=\{V_1,V_2,\ldots,V_s\}$. Clearly, $|\mathcal{F}|=c(G-A)$ and $|\mathcal{V}|=s$. It suffices to prove $|\mathcal{F}| \leq 3|\mathcal{V}|$. Moreover, let $\mathcal{F}_1:=\{F\in \mathcal{F}: V \subseteq V(F) \text{~for some~} V\in \mathcal{V}\}$ and $\mathcal{V}_1:=\{V\in \mathcal{V}: V\subseteq V(F) \text{~for some~} F \in \mathcal{F}\}$ (note that $V \neq V(G)$ here, and $V$ is just a representative element of $\mathcal{V}$). For each component $F \in \mathcal{F}$, there may be two sets $V,V' \in \mathcal{V}$ such that $V \subseteq V(F)$ and $V' \subseteq V(F)$, while for each set $V \in \mathcal{V}$, there is at most one component $F \in \mathcal{F}$ such that $V \subseteq V(F)$. Hence, $|\mathcal{F}_1| \leq |\mathcal{V}_1| \leq 3|\mathcal{V}_1|$. Let $\mathcal{F}_2:=\mathcal{F} \setminus \mathcal{F}_1$ and $\mathcal{V}_2:=\mathcal{V} \setminus \mathcal{V}_1$, that is, $\mathcal{F}_2:=\{F \in \mathcal{F}: V\nsubseteq V(F) \text{~for any~} V \in \mathcal{V}\}$ and $\mathcal{V}_2:=\{V \in \mathcal{V}: V\nsubseteq V(F) \text{~for any~} F \in \mathcal{F}\}$. Clearly, $|\mathcal{F}|=|\mathcal{F}_1|+|\mathcal{F}_2|$ and $|\mathcal{V}|=|\mathcal{V}_1|+|\mathcal{V}_2|$. It suffices to prove $$|\mathcal{F}_2| \leq 3|\mathcal{V}_2|.$$
	
	We define a bipartite graph $G[\mathcal{F},\mathcal{V}_2]$ as follows: its two parts are $\mathcal{F}$ and $\mathcal{V}_2$, and for $F \in \mathcal{F}$ and $V \in \mathcal{V}_2$, $F$ is adjacent to $V$ in $G[\mathcal{F},\mathcal{V}_2]$ if and only if $|V(F)\cap V|=1$ in $G$. Let $E_1$ denote the edges of $G[\mathcal{F},\mathcal{V}_2]$ incident to all vertices in $\mathcal{V}_2$, and $E_2$ denote the edges of $G[\mathcal{F},\mathcal{V}_2]$ incident to all vertices in $\mathcal{F}_2$. By Claim \ref{claim3.2}, each vertex in $\mathcal{V}_2$ is adjacent to two or three vertices in $\mathcal{F}$. Then $|E_1|\leq 3|\mathcal{V}_2|$. Clearly, for each $F\in \mathcal{F}$, $|V(F)|\geq 1$. By Claim \ref{claim3.1}, each vertex in $\mathcal{F}_2$ is adjacent to at least one vertex in $\mathcal{V}_2$. Then $|E_2|\geq |\mathcal{F}_2|$. Since $E_2\subseteq E_1$, we have
	\begin{equation*}
		|\mathcal{F}_2|\leq |E_2|\leq |E_1|\leq 3|\mathcal{V}_2|.
	\end{equation*}
	This proves the result as desired.
\end{proof}

	Let $H$ be a connected subcubic planar graph, and let $t \in \{2,3\}$ and $t \geq \Delta(H)$. The graph $G:=H\otimes K_t$ is obtained from $H$ and $|V(H)|$ copies of $K_t$ by replacing each vertex $v$ in $H$ with a $K_t$, and then joining the $d_H(v)$ edges of $H$ incident to $v$ to distinct vertices of the $K_t$, respectively (an edge is exactly incident to a vertex, and the operation guarantees planarity of $G$). One can see Fig. \hyperlink{Fig2}{2} for an illustration of this construction, where $H \cong K_4$ and $t=3=\Delta(H)$.
	
	\begin{figure}[h!]
		\begin{center}
			\begin{tikzpicture}[scale=.35]
				\tikzstyle{vertex}=[circle, draw, inner sep=0pt, minimum size=6pt]
				\tikzset{vertexStyle/.append style={rectangle}}
				\vertex (1) at (0,0) [scale=0.5,fill=black] {};
				\vertex (2) at (3,0) [scale=0.5,fill=black] {};
				\vertex (3) at (1.5,2.6) [scale=0.5,fill=black] {};
				\vertex (4) at (1.5,5.6) [scale=0.5,fill=black] {};
				\vertex (5) at (0,8.2) [scale=0.5,fill=black] {};
				\vertex (6) at (3,8.2) [scale=0.5,fill=black] {};
				\vertex (7) at (-2.6,-1.5) [scale=0.5,fill=black] {};
				\vertex (8) at (-5.6,-1.5) [scale=0.5,fill=black] {};
				\vertex (9) at (-4.1,-4.1) [scale=0.5,fill=black] {};
				\vertex (10) at (5.6,-1.5) [scale=0.5,fill=black] {};
				\vertex (11) at (8.6,-1.5) [scale=0.5,fill=black] {};
				\vertex (12) at (7.1,-4.1) [scale=0.5,fill=black] {};
				
				\path
				(1) edge (2)
				(2) edge (3)
				(3) edge (1)
				(1) edge (7)
				(7) edge (8)
				(8) edge (9)
				(9) edge (7)
				(4) edge (5)
				(5) edge (6)
				(6) edge (4)
				(10) edge (11)
				(11) edge (12)
				(12) edge (10)
				(2) edge (10)
				(3) edge (4)
				(5) edge (8)
				(6) edge (11)
				(9) edge (12)
				;
			\end{tikzpicture}
		\end{center}
		{\footnotesize \centerline{{\bf Fig. 2.}~ The graph $K_4\otimes K_3$.\hypertarget{Fig2}}}
	\end{figure}

\begin{proposition}\label{prop3.3}
	Let $H$ be a 3-edge-connected cubic planar graph. Then the graph $G=H\otimes K_3$ with order $n$ is 3-edge-connected, cubic and satisfies the equality 
	\begin{equation*}
	y_G=\frac{n}{2} -1=\frac{3\chi(G^c)}{2} -1.
	\end{equation*}
\end{proposition}

\begin{proof}
	By the definition of $G$, it is clear that $G$ is a connected cubic planar graph. For any two distinct vertices $u,v \in V(G)$, it is easy to check that there are three (pairwise) edge-disjoint $uv$-paths between $u$ and $v$ in $G$ since $H$ is 3-edge-connected. Hence, $G$ is also 3-edge-connected.
	
	\begin{claim}\label{claim3.3}
	$|V(H)|=\chi(G^c)$. 
	\end{claim}
	
	\begin{proof}
		Since the vertices of each $K_3$ in $G^c$ form an independent set, we can assign them with the same color. Thus, we have $\chi(G^c) \leq |V(H)|$. Suppose $\chi(G^c)< |V(H)|$. Since $n=|V(G)|=3|V(H)|$, there must exist a color class $V_i$ with $|V_i| \geq 4$ for some color $i \in [\chi(G^c)]$. Then, $G[V_i]$ is a complete graph of order at least 4. Since $G$ is cubic, $G \cong K_4$. However, $K_4$ can not be obtained by the definition of $G$ from any 3-edge-connected cubic planar graph $H$, a contradiction.
	\end{proof}
	
	If $\chi(G^c)=1$, then since $G$ is planar and 3-edge-connected, $G \cong K_4$. By the same reason as above, we may assume that $\chi(G^c)\geq 2$. Hence, $\left \lfloor \frac{3\chi(G^c)}{2}\right \rfloor-1 >1$, and by Theorem \ref{th3.2}, we have $y_G\leq \min\{\left \lfloor \frac{n}{2} \right \rfloor -1,\left \lfloor \frac{3\chi(G^c)}{2} \right \rfloor -1\}$.
	
	Finally, since $G$ is cubic and planar, by Theorems \ref{th1.4} and \ref{th2.6}, and Claim \ref{claim3.3}, we obtain
	\begin{equation*}
	 y_G=x_G=\frac{n}{2}-1=\frac{3|V(H)|}{2}-1=\frac{3\chi(G^c)}{2}-1.
	\end{equation*}
    This proves the proposition.	
\end{proof}

\noindent\textit{Remark.} For the complete graph $G \cong K_k$ with $k \geq 4$, $\chi(G^c) =1$, and $$\left \lfloor \frac{3\chi{(G^c)}}{2} \right \rfloor-1=0<1=y_G=1\leq \left \lfloor \frac{k}{2} \right \rfloor-1=\left \lfloor \frac{n}{2} \right \rfloor-1.$$
Combining with Proposition \ref{prop3.3}, we know that there are infinite many 3-edge-connected graphs attaining the corresponding bound of Theorem \ref{th3.2}. But for the case $\lambda \in \{1,2\}$, we fail to find any $\lambda$-edge-connected graph $G$ that can attain the bound related to the chromatic number of $G^c$.

\section{\large Further research}

Let $\omega(G)$ and $\alpha(G)$ be the clique number and independence number of a graph $G$. We say that $G$ is {\it perfect} if $\chi(H)=\omega(H)$ for each induced subgraph $H$ of $G$. If $G \in \{P_{n},C_{n}\}$ is of even order $n \geq 4$, then $G$ is bipartite, and thus, perfect. By the well-known {\it Perfect Graph Theorem}, $G$ is perfect if and only if $G^c$ is perfect, and we have $\chi(G^c)=\omega(G^c)=\alpha(G)=\frac{n}{2}$. The following proposition is obvious.

\begin{proposition}\label{prop4.1}
	If $G \in \{P_{n},C_{n}\}$ is of even order $n$, then
	\begin{displaymath}
	y_G= \left\{
	\begin{array}{ll}
		n=2\chi(G^c),& \text{if}\ G \cong P_{n},\\
		n-1=2\chi(G^c)-1,& \text{if}\ G \cong C_{n}.
	\end{array} \right.
\end{displaymath}
\end{proposition}

By Proposition \ref{prop4.1}, we see that the even path $P_{n}$ and even cycle $C_{n}$ attain the order-related bound of Theorem \ref{th3.2} for the case $\lambda \in \{1,2\}$. But, if $G \cong P_{n}$, then $y_G=2\chi(G^c)<3\chi(G^c)$, and if $G \cong C_{n}$, then $y_G=2\chi(G^c)-1<3\chi(G^c)-1$. Note that $P_{n}=P_{\frac{n}{2}} \otimes K_2$ and $C_{n}=C_{\frac{n}{2}} \otimes K_2$. Inspired by Proposition \ref{prop3.3}, we propose the following conjecture.

\begin{conjecture}
	If $G \ncong C_3$ is a $\lambda$-edge-connected (planar) graph with $\lambda \in \{1,2\}$, then $$y_G \leq 2\chi(G^c)-\lambda+1.$$
\end{conjecture}

\vspace{3mm}
\noindent{\large\bf Acknowledgments}
\vspace{3mm}

This work was supported by the National Natural Science Foundation of China (Nos. 12171402 and 12361070).

\end{document}